\documentclass{article}
\usepackage{amsfonts}
\usepackage{amsxtra}
\usepackage{amsthm}
\usepackage{amssymb}
\usepackage{verbatim}

\newtheorem{fed}{\textbf{Definition}}[section]
\newtheorem{thm}[fed]{\textbf{Theorem}}
\newtheorem{lemma}[fed]{\textbf{Lemma}}

\newtheorem{ex}[fed]{\textbf{Example}}
\newtheorem{rem}[fed]{\textbf{Remark}}
\newtheorem{prop}[fed]{\textbf{Proposition}}

\newtheorem{cor}[fed]{\textbf{Corollary}}
\usepackage{amscd,amssymb,bbm,graphicx,epsfig,psfrag,epic,eepic,latexsym,amsmath}
\usepackage{amsmath}
\usepackage[arrow,matrix,curve]{xy}
\usepackage{mathrsfs}
\usepackage{graphicx}
\usepackage{hyperref}
\usepackage{color}
\newcommand{\A}{\mathcal{A}}
\newcommand{\N}{\mathbb{N}}

\newcommand{\Q}{\mathbb{Q}}
\newcommand{\Z}{\mathbb{Z}}
\newcommand{\R}{\mathbb{R}}
\newcommand{\C}{\mathbb{C}}

\newcommand{\om}{\omega}
\newcommand{\Om}{\Omega}
\newcommand{\EBK}{\text{EBK}}
\newcommand{\p}{\partial}
\newcommand{\LL}{\mathcal{L}}
\newcommand{\la}{\langle}
\newcommand{\ra}{\rangle}
\newcommand{\into}{\hookrightarrow}

\begin{document}
\title{A variational characterization of Einstein--Brillouin--Keller quantization}
\author{Kai Cieliebak and Urs Frauenfelder}
\maketitle

\section{Introduction}

In 1917, Albert Einstein presented an intrinsic formulation of the
Bohr--Sommerfeld quantization conditions for arbitrary integrable systems~\cite{Einstein}.
In 1926, L\'eon Brillouin derived Einstein's conditions from the recently
found Schr\"odinger equation~\cite{Brillouin}.
In 1958, Joseph Keller rediscovered Einstein's conditions, enhanced
by a Maslov correction term involving half-integer quantum numbers~\cite{Keller,Keller-Rubinow}.
These enhanced conditions became known as {\em
Einstein--Brillouin--Keller (EBK) quantization}.

EBK quantization expresses the spectrum by quantizing the actions of
all invariant tori (see~\S\ref{ss:EBK}) and is defined only for integrable systems. On the
other hand, the Gutzwiller trace formula~\cite{gutzwiller} expresses
the spectrum in terms of periodic orbits, which also makes sense for
some non-integrable systems. This raises the question of how to
describe, in the integrable case, the EBK spectrum in terms of
periodic orbits. This question is addressed
e.g.~in~\cite{Balian-Bloch,Berry-Tabor} via complex periodic orbits.

In this paper we address this question via an object from symplectic
geometry, the {\em marked action spectrum} of a hypersurface $S$ in an
exact symplectic manifold (see~\S\ref{ss:toric}).  
This encodes the actions of periodic orbits on $S$ together with their
free homotopy classes. An old question in symplectic geometry asks 
to which extent a hypersurface is determined by its marked action spectrum,
generalizing the classical question to which extent a Riemannian
manifold is determined by its marked length spectrum.
In~\cite{Cieliebak-diss} a positive answer is given if $S$ is a 
level set of an integrable Hamiltonian system. Using this, we deduce
expressions of the EBK spectrum in terms of the marked action spectrum.
We will ignore the Maslov corrections for most of this paper and
indicate the required adjustments in~\S\ref{sec:Maslov}. 

To state the results, consider a {\em toric Hamiltonian}, i.e.~a
Hamiltonian $H:\C^n\to\R$ of the form
$$
   H(z) = f\Bigl(\frac12|z_1|^2,\dots,\frac12|z_n|^2\Bigr)
$$
for a smooth function $f:\R_+^n\to\R_+$. Assume in addition that $f$ is
homogeneous of some degree $d>0$, so that it is uniquely determined by
its regular level set $N=f^{-1}(1)\subset\R_+^n$. Our first result is 

\begin{thm}\label{thm:action-EBK}
In the setup above, assume in addition that the points with
nonvanishing Gauss curvature are dense in $N$. Then the EBK spectrum
of $H$ can be constructed from the marked action spectrum of the
hypersurface $H^{-1}(1)$. 
\end{thm}

The proof of this theorem gives explicit formulas for reconstructing
$N$ from the marked action spectrum (Proposition~\ref{prop:action-EBK}),
and for constructing the EBK spectrum from $N$ (Corollary~\ref{cor:action-EBK}).  

If $N$ is strictly convex or concave we get more explicit expressions.

\begin{thm}\label{thm:convex-toric}
In the setup above, assume in addition that $N$ is strictly convex.
Let $\A\subset\Z^n\times\R_+$ be the marked action spectrum of
$H^{-1}(1)$. Then the EBK spectrum of $H$ consists of the values
\begin{equation*}
  E_m = \Bigl(\sup_{(k,a)\in\A}\frac{\hbar\la m,k\ra}{a}\Bigr)^d,\qquad m\in\N_0^n.
\end{equation*}
For $N$ strictly concave an analogous formula holds with $\inf$
instead of $\sup$. 
\end{thm}

An interesting limiting case of this theorem is a system of $n$
uncoupled harmonic oscillators described by a linear function
$f(p)=\om_1p_1+\cdots+\om_np_n$. We explain in Example~\ref{ex:harm-cont}
how the correct spectrum can be recovered from
Theorem~\ref{thm:convex-toric} in this case.

{\bf Disk billiard. }
The second half of this paper studies in more detail a particular
integrable system, the billiard on the round disk $D_R\subset\R^2$ of
radius $R$. 
A direct computation (see~\cite{Weibert} or~\S\ref{sec:billiard}) shows that its EBK
spectrum consists of the values
\begin{equation*}
  E_{m,n} = \frac{\hbar^2F_{m,n}^2}{2R^2},\qquad m\in\Z,\ n\in\N_0,
\end{equation*}
where $F=F_{m,n}$ is the unique solution of the equation
\begin{equation*}
  \sqrt{F^2-m^2} - m\arccos(m/F) = \pi n.
\end{equation*}
On the other hand, V.\,Ramos has shown that the disk billiard can be
described by concave toric Hamiltonian and above~\cite{Ramos}. 
We verify in~\S\ref{sec:disk-toric} that applying
Theorem~\ref{thm:convex-toric} to this Hamiltonian recovers the energies
$E_{m,n}$ above. 

Our final result concerns the number theoretic properties of the EBK
spectrum of the disk billiard, or equivalently of the numbers $F_{m,n}$.

\begin{thm}\label{thm:schanuel}
All the numbers $F_{m,n}$ are transcendental. Moreover, if Schanuel's
conjecture (see~\S\ref{sec:Schanuel}) is true, then for each $n\in\N$ the set
$\{F_{m,n}\mid m \in \mathbb{N}_0\}$ is algebraically independent.
\end{thm}

\section{EBK quantization of toric domains}

\subsection{EBK quantization}\label{ss:EBK}

Let us start by recalling EBK quantization,
ignoring Maslov corrections which will be discussed in~\S\ref{sec:Maslov}. 

Consider an exact symplectic manifold $(V,\om=d\lambda)$ of dimension
$2n$ and an autonomous Hamilton function $H:V\to\R$. It gives rise to
the Hamiltonian vector field $X_H$ defined by $\om(\cdot,X_H)=dH$, the
Poisson bracket $\{F,G\}=\om(X_F,X_G)$, and
the Hamiltonian system $\dot x=X_H(x)$. We assume that this system is
{\em integrable}. This means that there exist $n$ integrals of motions
$P_1,\dots,P_n$ such that $\{P_i,P_j\}=0$ for all $i,j$ and the
differentials $dP_1,\dots,dP_n$ are linearly independent on an open
dense set $U\subset V$. We assume further that the function
$$
  P=(P_1,\dots,P_n):V\to\R^n
$$
is proper (i.e.~preimages of compact sets
are compact). By the Arnold--Liouville theorem~\cite{Arnold}, $U$ is
then fibered by invariant Lagrangian $n$-tori on which the Hamiltonian
flow is linear. Moreover, the Hamiltonian $H$ has the form
$$
  H = f(P_1,\dots,P_n) = f\circ P
$$
for a function $f:P(V)\to\R$. 
We call an invariant torus $L\subset V$ {\em quantized} if  
\begin{equation}\label{eq:quant}
  \int_\gamma\lambda\in 2\pi\hbar\Z\quad \text{for each loop $\gamma$
    on $L$,}
\end{equation}
where as usual $\hbar$ denotes Planck's constant divided by $2\pi$. 
Since $\lambda|_L$ is closed, is suffices to check this condition for
$n$ loops forming a $\Z$-basis of $H_1(L)$. 
The EBK approximation to the energy spectrum of the quantized system,
or simply the {\em EBK spectrum}, is then defined as
\begin{align*}
  \sigma_\EBK(H)
  &:= \{E\in\R\mid H^{-1}(E) \text{ contains a quantized invariant torus}\} \cr
  &= \{f(p)\mid \text{ the invariant torus $L_p=\{P=p\}$ is quantized}\}.
\end{align*}
The EBK spectrum has the following obvious properties:
\begin{description}
\item[(Scaling)]
For a smooth function $\phi:\R\to\R$,
$$
  \sigma_\EBK(\phi\circ H) = \phi\bigl(\sigma_\EBK(H)\bigr).
$$
\item[(Invariance)]
For a diffeomorphism $\psi:V\to V$ such that $\psi^*\lambda-\lambda$
is exact,
$$
  \sigma_\EBK(H\circ\psi) = \sigma_\EBK(H).
$$
\item[(Continuity)]
The map $H=f\circ P\mapsto\sigma_\EBK(H)$ is continuous with respect to the
strong $C^0$-distance on functions $f:P(V)\to\R$ and the Hausdorff distance
on subsets of $\R$. 
\end{description}

\subsection{Toric Hamiltonians}\label{ss:toric}

Now we specialize to our main class of examples. Consider $\C^n$ with
coordinates $z_j=x_j+iy_j$ and the standard Liouville form
$\lambda=\sum_jx_jdy_j$. A {\em toric Hamiltonian} is a Hamiltonian
$H:\C^n\to\R$ of the form
$$
   H(z) = f\Bigl(\frac12|z_1|^2,\dots,\frac12|z_n|^2\Bigr)
$$
for a smooth function $f:\R_+^n\to\R_+$.
The sublevel set $\{H\leq 1\}\subset\C^n$ is known as a {\em toric domain}.
The corresponding Hamiltonian system $\dot z=i\nabla H(z)$ has the solutions
$$
  z_j(t) = z_j(0)e^{i\p_j\hspace{-1pt}f(p)t},
$$
where $\p_jf(p)$ denotes the $j$-th partial derivative of $f$ at the
point $p=(\frac12|z_1|^2,\dots,\frac12|z_n|^2)$.
So the system is integrable with integrals $P_j=\frac12|z_j|^2$. 
Note that $(P_j,\phi_j)$, where $\phi_j$ is the argument of $z_j$, are
action-angle coordinates.  
Each $p=(p_1,\dots,p_n)\in\R_+^n$ defines an invariant torus
$$
   L_p = \Bigl\{z\in\C^n\;\Bigl|\ \frac12|z_j|^2=p_j\text{ for }j=1,\dots,n\Bigr\}.
$$
The loops 
$$
  \gamma_j(t)=(\sqrt{2p_1},\dots,\sqrt{2p_j}e^{it},\dots,\sqrt{2p_n}),\qquad t\in[0,2\pi]
$$
generate $H_1(L_p)$ and have action $\int_{\gamma_j}\lambda = 2\pi p_j$, 
so $L_p$ is quantized iff  
$$
  p_j = \hbar m_j\quad\text{with }m_j\in\N_0,\quad j=1,\dots,n.
$$
Hence, the EBK spectrum is given by
\begin{equation}\label{eq:toric-EBK}
  \sigma_\EBK = \{f(\hbar m)\mid m\in\N_0^n\}.
\end{equation}
Note that a solution $z$ of the Hamiltonian system is $T$-periodic iff
$\p_j\hspace{-1pt}f(p)T=2\pi k_j$ for integers $k_1,\dots,k_n$, 
so the flow on a torus $L_p$ is $T$-periodic iff
$$
  \frac{T}{2\pi}\nabla f(p) = k\in\Z^n. 
$$
The $T$-periodic solutions $z$ on such a {\em rational torus} have action
$$
  \int_z\lambda = 2\pi\la p,k\ra.
$$
We define the {\em marked action spectrum} of the energy hypersurface
$H^{-1}(1)$ by
$$
  \A := \{(k,\la p,k\ra) \mid k\in\Z^n,\,f(p)=1,\,\nabla f(p)\sim k\}.
$$
It records the actions (up to the factor $2\pi$) of orbits on rational
tori together with their homology classes in $H_1(T^n)=\Z^n$. 
Using the unit normal vector $n(p)=\nabla f(p)/|\nabla
f(p)|$, we can write it in terms of the level set $N=f^{-1}(p)$ as  
$$
  \A = \{(k,\la p,k\ra) \mid k\in\Z^n,\,p\in N,\,n(p)\sim k\}.
$$

\begin{ex}[Uncoupled harmonic oscillators]\label{ex:harm}
A system of $n$ uncoupled harmonic oscillators of angular frequencies
$\om_1,\dots,\om_n$ is described by the Hamilton function
$$
  H(z) = \sum_{j=1}^n\frac{\om_j}{2}|z_j|^2.
$$
This is a toric domain as above with the linear function
$f(p)=\om_1p_1+\dots+\om_np_n$, so its EBK energy spectrum is
$$
  \sigma_\EBK = \Bigl\{\sum_{j=1}^n\hbar\om_j m_j \mid m\in\N_0^n\Bigr\}.
$$
This agrees with the quantum mechanical spectrum up to replacing
$m_j$ by $m_j+\frac12$.
Note that for $n\geq 2$ and rationally independent frequencies
$\om_1,\dots,\om_n$ the flow on each $n$-torus $L_p$ is irrational
(i.e.~its orbits are dense), so the
formula for the EBK energy spectrum arises entirely from irrational tori.
The only periodic orbits of energy $E$ in this case are the $n$ orbits
where all but one coordinate is zero.
We will see in Example~\ref{ex:harm-cont} how the EBK spectrum can nonetheless be
derived from the marked action spectrum by an approximation process. 
A direct expression of the EBK spectrum in terms of periodic orbits 
is given in~\cite{Cieliebak-Frauenfelder-rabquant} using Tate Rabinowitz Floer homology. 
\end{ex}

\section{Legendre transformations}\label{sec:Leg}

In this section we describe Legendre-type transformations in three
different settings and discuss their properties and relationships. 
\smallskip

\subsection{Strictly convex functions}\label{ss:convex}

The {\em Legendre transform} of a strictly convex continuous function
$f:\R^n\to\R$ is the function $\LL f:\R^n\to\R$ defined by 
$$
  \LL f(q) := \sup_{p\in\R^n}\Bigl(\la p,q\ra - f(p)\Bigr). 
$$
It is well-known (see e.g.~\cite{Arnold}) that $\LL f$ is again
strictly convex and $\LL(\LL f)=f$. 
If $f$ is of class $C^2$ and the matrix $d^2f(p)$ of second
derivatives is positive definite for each $p$, then
$$
  \LL f(q) = \la p_0,q\ra - f(p_0) = \la(\nabla f)^{-1}(q),q\ra -
  f\bigl((\nabla f)^{-1}(q)\bigr)
$$
for the unique $p_0\in\R^n$ with $\nabla f(p_0)=q$. 
\smallskip

\subsection{Homogeneous convex functions}\label{ss:hom}

Suppose now that $f:\R^n\to\R$ is smooth and homogeneous of degree $1$, 
$$
  f(tp) = tf(p)\quad\text{for all }t>0,
$$
and its level set $N:=f^{-1}(1)$ is regular and strictly convex
(i.e.~all its normal curvatures are positive, where $N$ is cooriented
by $\nabla f$). Then $f$ is only weakly convex and the Legendre
transform from~\S\ref{ss:convex} is not applicable. Instead, we define its {\em
  homogeneous Legendre transform} $Lf:\R^n\to\R$ by
$$
  Lf(q) := \sup_{p\in N}\Bigl(\la p,q\ra - f(p)\Bigr) + 1
  = \sup_{p\in N}\la p,q\ra.
$$
It follows that 
$$
  Lf(q) = \la p_0,q\ra = \la n^{-1}(q/|q|),q\ra
$$
for the unique $p_0\in N$ with $n(p_0)\sim q$, where $\sim$
denotes positive proportionality and $n$ is the Gauss map
$$
  n:N\to S^{n-1},\qquad p\mapsto\frac{\nabla f(p)}{|\nabla f(p)|}.
$$
\smallskip

\subsection{Hypersurfaces}\label{ss:hyp}

Since the homogeneous function $f$ in~\S\ref{ss:hom} and its level set
$N=f^{-1}(1)$ determine each other, one can reformulate the
homogeneous Legendre transform in terms of the hypersurface $N$.
This transformation has been introduced in greater generality
in~\cite{Cieliebak-diss} and we recall here the relevant definitions
and results.

Consider a compact smooth hypersurface $N\subset\R^n$ with Gauss map
$n:N\to S^{n-1}$ associating to $p\in N$ its outward pointing unit
normal vector $n(p)$. The {\em Gauss curvature} at $p\in N$ is given
by $K(p):=\det Dn(p)$. We do not require $N$ to be convex.
For $p\in N$ with $\la p,n(p)\ra\neq 0$ we define
\begin{equation}\label{eq:L}
  L(p) := \frac{n(p)}{\la p,n(p)}.
\end{equation}

\begin{lemma}[\cite{Cieliebak-diss}]\label{lem:Kai}
Assume that the points with nonvanishing Gauss curvature are dense in $N$.
Then the set of {\em nice points} 
$$
  N' := \bigl\{p\in N\mid K(p)\neq 0,\,\la p,n(p)\ra\neq 0,\,L^{-1}(L(p))=\{p\}\bigr\}
$$
is dense in $N$ and the map~\eqref{eq:L} defines an embedding $L:N'\into\R^n$.
Its image $L(N')$ consists again of nice points and $L(L(p))=p$ for each $p\in N'$. 
Thus $L(L(N'))=N'$, so the hypersurface $N$ can be
reconstructed from its {\em Legendre transform} $L(N')$ by
$$
  N = \overline{L(L(N'))}. 
$$
\end{lemma}

\begin{lemma}[\cite{Cieliebak-diss}]\label{lem:Kai-convex}
Let $N\subset\R^n$ be a smooth, compact, strictly convex hypersurface
enclosing the origin. Then $N=N'$ consists of nice points, $L(N)$ is
again a smooth, compact, strictly convex hypersurface
enclosing the origin, and 
$$
  L(N) = \bigl\{q\in\R^n\mid \sup_{p\in N}\la p,q\ra=1\}. 
$$
\end{lemma}
  
\begin{proof}
Everything except the last assertion is proved in~\cite{Cieliebak-diss}.
The last assertion holds because $\sup_{p\in N}\la p,q\ra = \la p_0,q\ra$
for the unique $p_0\in N$ with $n(p_0)\sim q$, thus $\sup_{p\in N}\la
p,q\ra = \la p_0,q\ra=1$ iff $q=\frac{n(p_0)}{\la p_0,n(p_0)\ra}\in L(N)$. 
\end{proof}

Lemmas~\ref{lem:Kai} and~\ref{lem:Kai-convex} immediately imply

\begin{cor}\label{cor:Kai-convex}
Consider a $1$-homogeneous function $f:\R^n\to\R$ with strictly
convex compact level set $N=f^{-1}(1)$ and its homogeneous Legendre
transform $Lf$ as in~\S\ref{ss:hom}. Then 
$$
  L(N) = (Lf)^{-1}(1), 
$$
$Lf:\R^n\to\R$ is again $1$-homogeneous with strictly convex compact level
sets, and
$$
  L(Lf)=f. 
$$
\end{cor}

\begin{rem}
It immediately follows from the proofs that Lemma~\ref{lem:Kai-convex}
and Corollary~\ref{cor:Kai-convex} remain true with ``convex''
replaced by ``concave'' and ``sup'' replaced by ``inf'' in the
definitions of $L(N)$ and $Lf$. 
\end{rem}

\section{Recovering the EBK spectrum from the marked action spectrum}\label{sec:action-EBK}

\subsection{General toric domains}

Consider a hypersurface $N\subset\R^n$. Recall from~\S\ref{ss:toric}
its marked action spectrum
$$
  \A = \{(k,\la p,k\ra) \mid k\in\Z^n,\,p\in N,\,n(p)\sim k\}.
$$

\begin{prop}[\cite{Cieliebak-diss}]\label{prop:action-EBK}
Assume that the points with nonvanishing Gauss curvature are dense in $N$.
Then $N$ can be recovered from its marked action spectrum as the closure of
the Legendre transform
$$
  N = \overline{L(M')},
$$
where $M'$ is the set of nice points in
$$
  M := \overline{\Bigl\{\frac{k}{a}\;\Bigl|\;(k,a)\in\A,\,a\neq 0\Bigr\}}.
$$
\end{prop}
  
\begin{proof}
For the reader's convenience let us recall the proof from~\cite{Cieliebak-diss}.
For $p\in N$ with $n(p)\sim k\in\Z^n$ and $a=\la p,k\ra\neq 0$ we have
$\frac{k}{a}=\frac{n(p)}{\la p,n(p)\ra}=L(p)$. Thus
$$
  M = \overline{\{L(p)\mid p\in N'\text{ rational }\}} = \overline{L(N')},
$$
where $p\in N'$ is called rational if $Tn(p)\in\Z^n$ for some $T>0$,
and the last equality follows because the rational points are dense in the
set $N'$ of nice points. 
Let $M'\subset M$ be the set near which $M$ is a nice hypersurface.
Then $M'=L(N')$ and therefore $N = \overline{L(M')}$ by Lemma~\ref{lem:Kai}. 
\end{proof}

Let $f:\R^n\to\R$ be the $1$-homogeneous function with level set $N=f^{-1}(1)$, 
and $\sigma_\EBK$ the EBK spectrum of the associated Hamilton function
$H(z)=f(|z_1|^2/2,\dots,|z_n|^2/2)$. According to~\eqref{eq:toric-EBK}
we have $E\in\sigma_\EBK$ iff $E=f(\hbar m)$ for some $m\in\N_0^n$, or equivalently iff
$$
  \frac{\hbar m}{E} = \frac{\hbar m}{f(\hbar m)}\in N.
$$
Thus Proposition~\ref{prop:action-EBK} implies

\begin{cor}\label{cor:action-EBK}
For $N$ as in Proposition~\ref{prop:action-EBK}, the EBK spectrum of
the associated $1$-homogeneous Hamiltonian can be
recovered from its marked action spectrum by recovering $N$ as described in
Proposition~\ref{prop:action-EBK} and writing
$$
  \sigma_\EBK = \Bigl\{E\in\R\;\Bigl|\; \text{there exists }m\in \N_0^n\text{ such
    that }\frac{\hbar m}{E}\in N\Bigr\}.
$$
\end{cor}

\subsection{Convex and concave toric domains}\label{ss:convex-toric}

In the case that the hypersurface $N\subset\R^n$ is strictly convex or
concave the EBK spectrum can be given by an explicit formula in terms
of the marked action spectrum.
For $\ell\in\N_0$ we denote 
$$
  \N_\ell^n:=\{k\in\Z^n \mid k_j\geq\ell \text{ for all $j$ and there
      exists $p\in N$ with $n(p)\sim k$}\}.
$$
For $k\in\N_\ell^n$ we set
$$
  a(k) := \la p,k\ra \text{ for the unique $p\in N$ with $n(p)\sim k$}. 
$$

\begin{prop}\label{prop:convex-toric}
Let $N\subset\R^n$ be a strictly convex compact hypersurface and
$\A\subset\Z^n\times\R_+$ its marked action spectrum. Then the EBK
spectrum $\sigma_\EBK$ of the associated $1$-homogeneous Hamiltonian
consists of the values
\begin{equation}\label{eq:Em}
  E_m = \sup_{(k,a)\in\A}\frac{\hbar\la m,k\ra}{a}
  = \sup_{k\in\N_0^n}\frac{\hbar\la m,k\ra}{a(k)},\qquad m\in\N_0^n.
\end{equation}
Equivalently, $E_m$ can be characterized as the unique number such that
\begin{equation}\label{eq:Em-minmax}
  \sup_{\ell\in\N}\inf_{k\in\N_\ell^n}\Bigl(Ea(k)-\hbar\la m,k\ra\Bigr)
  = \begin{cases}
    +\infty & \text{for }E>E_m, \cr
    -\infty & \text{for }E<E_m.
  \end{cases}
\end{equation}
For $N$ strictly concave analogous formulas hold with $\sup$ and
$\inf$ exchanged.
\end{prop}

\begin{proof}
The EBK spectrum $\sigma_\EBK$ consists of the energies $E_m$,
$m\in\N_0^n$, that are uniquely characterized by the following
equivalent conditions:
$$
\frac{\hbar m}{E_m}\in N = \overline{L(M')}
\Longleftrightarrow \sup_{q\in M}\Bigl\la\frac{\hbar m}{E_m},q\Bigr\ra = 1
\Longleftrightarrow \sup_{(k,a)\in\A}\Bigl\la\frac{\hbar m}{E_m},\frac{k}{a}\Bigr\ra = 1.
$$
Here the first characterization follows from
Proposition~\ref{prop:action-EBK} and Corollary~\ref{cor:action-EBK},
the first equivalence from Lemma~\ref{lem:Kai-convex}, 
and the second equivalence from the definition on $M$ in Proposition~\ref{prop:action-EBK}.
Solving the last condition for $E_m$ yields formula~\eqref{eq:Em}. 

To prove~\eqref{eq:Em-minmax}, note that for $k\in\N_\ell^n$ we have
$a(k)=\la p,k\ra \geq \max_jp_jk_j\geq c\ell$ with $c:=\min_{p\in N}\max_jp_j>0$.
Consider now $E>E_m$. Then for each $\ell\in\N_0$ we have
$$
  \inf_{k\in\N_\ell^n}\Bigl(E-\frac{\hbar\la m,k\ra}{a(k)}\Bigr) =
  E-\sup_{k\in\N_\ell^n}\frac{\hbar\la m,k\ra}{a(k)} = E-E_m > 0,
$$
where the last equality follows by replacing $k$ by large integer
multiples and making all its components nonzero. Combining these we get
$$
  \inf_{k\in\N_\ell^n}\Bigl(Ea(k)-\hbar\la m,k\ra\Bigr)
  = \inf_{k\in\N_\ell^n}\Bigl(a(k)\bigl(E-\frac{\hbar\la m,k\ra}{a(k)}\bigr)\Bigr)
  \geq c\ell(E-E_m),
$$
which becomes $+\infty$ in the supremum over $\ell\in\N$. This proves
the case $E>E_m$ in~\eqref{eq:Em-minmax}, and the case $E<E_m$ is analogous.
\end{proof}

\begin{ex}[Uncoupled harmonic oscillators (continued)]\label{ex:harm-cont}
In the notation of Example~\ref{ex:harm}, suppose that the frequency
vector $\om=(\om_1,\dots,\om_n)$ is {\em rational}, i.e.~$T\om\in\N_0^n$ for some $T>0$. 
Then all orbits are periodic and Proposition~\ref{prop:convex-toric}
is still applicable. Thus $\sigma_\EBK$ consists of the values
$E_m = \sup_{k\in\N_0^n}\frac{\hbar\la m,k\ra}{a(k)}$ where $k=\ell
T\om$ for some $\ell\in\N$ and $a(k)=\la p,k\ra = \ell T\la p,\om\ra =
\ell T$ for some (hence any) $p\in N$. This implies $E_m=\hbar\la
m,\om\ra$ and we recover the EBK spectrum from Example~\ref{ex:harm}.
Continuity of the spectrum now implies that this formula for $\sigma_\EBK$ 
continues to hold for {\em all} (not necessarily rational) frequency vectors,
although Proposition~\ref{prop:convex-toric} is not applicable to such
$\om$ due to the lack of rational invariant tori.
\end{ex}

So far in this section we have assumed that $f:\R^n\to\R$ is
$1$-homogeneous. If $f$ is instead $d$-homogeneous for some $d>0$,
then $\bar f=f^{1/d}$ is $1$-homogeneous with the same level set $\bar
f^{-1}(1)=f^{-1}=N$. By the (Scaling) property of the EBK spectrum, 
Corollary~\ref{cor:action-EBK} and Proposition~\ref{prop:convex-toric}
thus carry over to the $d$-homogeneous case by replacing the EBK
energies given there by their $d$-th powers. In particular, we obtain
Theorems~\ref{thm:action-EBK} and~\ref{thm:convex-toric} from the
Introduction.

\section{Billiard on the disk}\label{sec:billiard}

A well studied integrable system is the billiard on the closed disk
$D_R\subset\R^2$ of radius $R>0$; see e.g.~\cite{Weibert} and the
references therein for its classical and quantum mechanical treatment.
In this section we indicate how to derive its EBK spectrum and verify 
Proposition~\ref{prop:convex-toric} on this example. 

Without loss of generality we set the mass of the billiard ball to $1$,
so that the system is described by the free Hamiltonian
$$
  H(q,p)=\frac12|p|^2.
$$
An invariant torus $L$ in an energy hypersurface $H^{-1}(E)$ consists of
all billiard trajectories of speed $\sqrt{2E}$ with a given reflection
angle $\alpha\in(0,\pi)$ with the boundary. These trajectories are tangent to
the circle of radius $S=R|\cos\alpha|$. They are periodic if
\begin{equation}\label{eq:alpha-quant}
  \alpha = \frac{k\pi}{\ell}\qquad\text{with }k\in\Z,\,\ell\in\N,
\end{equation}
and otherwise the are dense in the annulus $A=\{S\leq|q|\leq R\}$.
The projection $(q,p)\to q$ defines a map $L\to A$ which is 2-1 over
the interior of $A$, corresponding to the two possible values of $p$ for
trajectories of $L$ through a point $q$ (cf.~\cite{Einstein}). 

Let us first assume that $\alpha\leq\pi/2$. 
The first homology of $L$ is generated by two loops: the loop
$\gamma$ running once around the inner circle given by
$$
  q(t)=R\cos\alpha e^{it},\ p(t)=i\sqrt{2E}e^{it},\qquad t\in[0,2\pi],
$$
and the loop $\delta$ running along a straight line from the inner
circle to the outer one with one lift and back with the other lift
whose first half is given by 
$$
  q(t)=t,\ p(t)=\sqrt{2E}e^{i\arcsin(R\cos\alpha/t)},\qquad t\in[S,R].
$$
Short computations with $\lambda=p\,dq$ yield the action integrals
$$
   \int_{\gamma}\lambda = 2\pi R\sqrt{2E}\cos\alpha,\qquad
   \int_{\delta}\lambda = 2R\sqrt{2E}(\sin\alpha-\alpha\cos\alpha).
$$
The EBK quantization conditions are
$$
  \int_{\gamma}\lambda=2\pi\hbar m,\quad\text{and}\quad
  \int_{\delta}\lambda=2\pi\hbar n,\qquad m,n\in\N_0,
$$
where $m,n\geq 0$ because $\cos\alpha\geq 0$ and
$\sin\alpha-\alpha\cos\alpha\geq 0$ for all $\alpha\in[0,\pi/2]$. Abbreviating
$$  
  F := R\sqrt{2E}/\hbar,
$$
these conditions become
$$
  F\cos\alpha = m\quad\text{and}\quad F(\sin\alpha-\alpha\cos\alpha)=\pi n,\qquad m,n\in\N_0
$$
and they give rise to the equation
\begin{equation}\label{eq:billiard-EBK}
  \sqrt{F^2-m^2} - m\arccos(m/F) = \pi n.
\end{equation}
An angle $\alpha\in(\pi/2,\pi)$ leads to the same quantization
condition as the angle $\pi-\alpha\in(0,\pi/2)$.  
The EBK spectrum thus consists of the values
\begin{equation}\label{eq:E-F}
  E_{m,n} = \frac{\hbar^2F_{m,n}^2}{2R^2},
\end{equation}
where $F_{m,n}$ is the unique\footnote{See~\S\ref{sec:transcendence}
for the proof of uniqueness.} solution
of~\eqref{eq:billiard-EBK} for given $m,n\in\N_0$ and each energy
value with $m>0$ appears with multiplicity $2$. 
Writing $F=kR$, this agrees with formula (4.19) in~\cite{Weibert} up
to replacing $n$ by $n+3/4$.

A short computation shows that
the periodic orbits of energy $E$ with angle $\alpha$
satisfying~\eqref{eq:alpha-quant} have action
\begin{equation}\label{eq:action-billiard}
  a(k,\ell) = 2R\sqrt{2E}\ell\sin\Bigl(\frac{\pi k}{\ell}\Bigr).
\end{equation}
We will see in the next section how the EBK spectrum can be derived
from these action values.

\section{Disk billiard as a toric domain}\label{sec:disk-toric}

In the beautiful paper~\cite{Ramos} V.~Ramos explains how the billiard
on the disk can be interpreted as a concave toric domain. For this, he
first describes the billiard on the unit disk $D_1\subset\R^2$ as the
flow on the boundary of the Lagrangian bidisk $D_1\times D_1\subset T^*\R^2=\R^4$.
Then he proves (\cite[Theorem\,3]{Ramos}) that the interior of the
Lagrangian bidisk $D_1\times D_1$ is symplectomorphic to the concave
toric domain
\begin{equation}\label{eq:XOm}
   X_\Om = \Bigl\{(z_1,z_2)\in\C^2\;\Bigl|\;
   \Bigl(\frac{\pi}{2}|z_1|^2,\frac{\pi}{2}|z_2|^2\Bigr)\in\Om\Bigr\},  
\end{equation}
where $\Om\subset[0,\pi]^2$ is the domain bounded by the coordinate
axes and the curve $N\subset[0,\pi]^2$ parametrized by\footnote{
To simplify the formulas we have replaced the angle $\alpha$
in~\cite{Ramos} by $2\alpha$ and denoted the domain $\Om_0$
in~\cite{Ramos} by $2\Om$.}
$$
  \rho(\alpha) = \Big(\sin(\alpha)-\alpha \cos (\alpha),\sin(\alpha)
  + (\pi-\alpha)\cos(\alpha)\Big),\qquad \alpha\in[0,\pi].
$$
To determine the action spectrum, we compute
\begin{eqnarray*}
  \rho'(\alpha)
&=&\Big(\alpha \sin(\alpha),(\alpha-\pi)\sin(\alpha)\Big).
\end{eqnarray*}
So the normal vector $n(p)$ at a point $p=(p_1,p_2)=\rho(\alpha)\in N$ satisfies
$$
  n(p) \sim \Big((\pi-\alpha)\sin(\alpha),\alpha \sin(\alpha)\Big)
  \sim \Big(\pi-\alpha,\alpha\Big).
$$
The correponding orbits are periodic iff $n(p)\sim(k_1,k_2)\in\N_0^2$,
which can be solved for
$$
  \alpha=\frac{k_2\pi}{k_1+k_2}\qquad\text{and}\qquad \pi-\alpha=\frac{k_1\pi}{k_1+k_2}.
$$
The corresponding action therefore becomes
\begin{eqnarray*}
  a(k_1,k_2) &=& k_1p_1+k_2p_2\\
  &=& k_1\Big(\sin(\alpha)-\alpha \cos(\alpha)\Big)+k_2
  \Big(\sin(\alpha)+(\pi-\alpha)\cos(\alpha)\Big)\\
  &=& (k_1+k_2)\sin\bigg(\frac{\pi k_2}{k_1+k_2}\bigg)\\
  &=&(k_1+k_2)\sin\bigg(\frac{\pi}{2}\bigg(1+\frac{k_2-k_1}{k_1+k_2}\bigg)\bigg)\\
  &=&(k_1+k_2)\cos\bigg(\frac{\pi(k_2-k_1)}{2(k_1+k_2}\bigg).
\end{eqnarray*}
Note that these actions agree up to a scaling factor with those for
the disk billiard given by~\eqref{eq:action-billiard} with $k=k_2$ and
$\ell=k_1+k_2$. 

The EBK spectrum can now be determined from these action values via
Proposition~\ref{prop:convex-toric}. We leave the straightforward
computation to the reader and instead determine the EBK spectrum via
Corollary~\ref{cor:action-EBK} (which we know yields the same result).
Thus $E\in\sigma_\EBK$ iff there exists $(m_1,m_2)\in\N_0^2$ such that 
$\frac{\hbar}{E}(m_1,m_2)\in N$, i.e.
$$
   m_1=\frac{E}{\hbar}(\sin\alpha-\alpha\cos\alpha),\qquad
   m_2=\frac{E}{\hbar}\bigl(\sin\alpha+(\pi-\alpha)\cos\alpha\bigr)
$$
for some $\alpha\in[0,\pi]$. Setting
$$
  F := \frac{E\pi}{\hbar},
$$
the difference of the two equations yields $m_2-m_1=F\cos\alpha$,
hence
$$
  \alpha=\arccos\bigg(\frac{m_2-m_1}{F}\bigg).
$$
Inserting this into the equation for $m_1$ gives
\begin{eqnarray*}
  \pi m_1 &=& F(\sin\alpha-\alpha\cos\alpha) \\
  &=& F\biggl(\sqrt{1-\frac{(m_2-m_1)^2}{F^2}} -
  \frac{m_2-m_1}{F}\arccos\Bigl(\frac{m_2-m_1}{F}\Bigr)\biggr) \\
  &=& \sqrt{F^2-(m_2-m_1)^2} - (m_2-m_1)\arccos\Bigl(\frac{m_2-m_1}{F}\Bigr).
\end{eqnarray*}
With $n=m_1$ and $m=m_2-m_1$ this agrees with equation~\eqref{eq:billiard-EBK}
for the EBK spectrum of the disk billiard. Note that the actual
energies $E=E_{m,n}$ of the billiard on the unit disk are obtained from
the unique solutions $F=F_{m,n}$ of~\eqref{eq:billiard-EBK} using the
relation~\eqref{eq:E-F} with $R=1$, i.e.~$E=\frac{\hbar^2F^2}{2}$, rather than
the relation $E=\frac{\hbar F}{\pi}$ above. The additional factor
$\pi$ arises because in~\eqref{eq:XOm} we have defined the toric
domain $X_\Om$ using $\frac{\pi}{2}|z_j|^2$ rather than the action
coordinates $\frac{1}{2}|z_j|^2$. The remaining discrepancy is due to
the billiard Hamiltonian $\frac{1}{2}|p|^2$ and the $1$-homogeneous
Hamiltonian $|p|$ which we need to use in Corollary~\ref{cor:action-EBK}
and Proposition~\ref{prop:convex-toric}.

\section{Transcendental numbers and Schanuel's conjecture}\label{sec:Schanuel}

In this section we recall some definitions and facts about
transcendental field extensions and Schanuel's conjecture, see
e.g.~\cite[Chapter 23]{karpfinger-meyberg} and~\cite[Chapter 21]{murty-rath}. 

Consider a field extension $K\subset L$. 
A finite subset $\{a_1,\dots,a_N\} \subset L$ is called {\em
  algebraically independent} over $K$ if there exists no nontrivial
polynomial $p \in K[x_1,\ldots x_N]$ in $N$ variables with
coefficients in $K$ such that
$$p(a_1,\ldots,a_N)=0.$$
An arbitrary subset $A \subset L$ is called algebraically independent
if each finite subset of $A$ is algebraically independent. 
A {\em transcendence basis} of $L$ over $K$ is a maximal
algebraically independent subset of $L$. Transcendence bases exist and
all have the same cardinality, which is called the \emph{transcendence
degree} of $L$ over $K$ and denoted 
$$\mathrm{trdeg}(L/K).$$
Thus the field extension is algebraic iff $\mathrm{trdeg}(L/K)=0$, and
$\{a_1,\dots,a_N\} \subset L$ is algebraically independent iff
$\mathrm{trdeg}(K(a_1,\dots,a_N)/K)=N$. Here $K(A)\subset L$ denotes
the smallest subfield containing $K$ and a subset $A\subset L$.   
A basic property of the transcendence degree is its additivity for
field extensions $K\subset L\subset M$,
\begin{equation}\label{eq:degree}
  \mathrm{trdeg}(M/K) = \mathrm{trdeg}(M/L) + \mathrm{trdeg}(L/K).
\end{equation}
We will need the following easy consequence of this formula.

\begin{lemma}\label{lem:trdeg}
Let $K\subset L$ be a field extension and $A,B\subset L$ be subsets
such that each element of $A$ is algebraic over $K(B)$ and vice
versa. Then
$$
  \mathrm{trdeg}(K(A)/K) = \mathrm{trdeg}(K(B)/K).
$$
\end{lemma}

\begin{proof}
Apply~\eqref{eq:degree} to the field extensions $K\subset K(A)\subset
K(A\cup B)$ and $K\subset K(B)\subset K(A\cup B)$, noting that
$K(A)\subset K(A\cup B)$ and $K(B)\subset K(A\cup B)$ are algebraic
and thus have transcendence degree zero. 
\end{proof}

From now on we will specialize to field extensions $K=\Q\subset
L\subset \C$ and drop the specification ``over $\Q$'',
writing $\mathrm{trdeg}(L)$ for $\mathrm{trdeg}(L/\Q)$ etc. We will need
the following special case of the Lindemann--Weierstrass Theorem (see
e.g.~\cite[Corollary 4.2]{murty-rath}).

\begin{thm}\label{thm:Lindemann-Weierstrass}
If $0\neq \alpha\in\C$ is algebraic, then $e^a$ is transcendental. 
\end{thm}

This corresponds to the case $N=1$ of the following conjecture.

\textbf{Schanuel's Conjecture. }
\emph{If $\alpha_1,\ldots \alpha_N\in\C$ are linearly independent over
$\mathbb{Q}$, then} 
$$
  \mathrm{trdeg}\,\mathbb{Q}\big(\alpha_1,\ldots \alpha_N, e^{\alpha_1},\ldots,
  e^{\alpha_N}\big) \geq N.
$$
We quote from the book of Ram Murty and Rath~\cite[Chapter 21]{murty-rath}:
\begin{quote}
``This conjecture is believed to include all known transcendence results as well
as all reasonable transcendence conjectures on the values of the exponential function."
\end{quote}

\section{Transcendence of the EBK spectrum of the disk billiard}\label{sec:transcendence}

We fix a positive rational number $n \in \mathbb{Q} \cap (0,\infty)$. We denote by
$$\arccos \colon [-1,1] \to [0,\pi]$$
the inverse function of $\cos|_{[0,\pi]}$. 
For $m \in \mathbb{N}_0$ we consider the real numbers
$$F_m:=F_{m,n} \in \mathbb{R}$$ 
defined implicitly by equation~\eqref{eq:billiard-EBK}, 
\begin{equation}\label{edef}
  \sqrt{F^2_m-m^2}-m\arccos\bigg(\frac{m}{F_m}\bigg)=n\pi.
\end{equation}
That this equation has a unique solution can be seen as follows. For
$m \geq 0$ consider the function
$$\mathfrak{f}_m \colon [m,\infty)\to\R, \qquad x \,\mapsto
\sqrt{x^2-m^2}- m \arccos\bigg(\frac{m}{x}\bigg).$$
Its derivative is given by
$$\mathfrak{f}_m'(x)=\frac{x}{\sqrt{x^2-m^2}}- \frac{m^2}{x^2\sqrt{1-\frac{^2}{x^2}}}=\frac{x^2-m^2}{x\sqrt{x^2-m^2}}=\frac{\sqrt{x^2-m^2}}{x},$$
so $\mathfrak{f}_m$ is strictly increasing. Moreover, it satisfies
$$\mathfrak{f}_m\big(m)=0, \qquad \lim_{x \to \infty} \mathfrak{f}_m(x)=\infty$$
and thus gives rise to a bijection
$$\mathfrak{f}_m \colon [m,\infty) \to [0,\infty).$$
Therefore, we can set
$$F_m=\mathfrak{f}_m^{-1}(n).$$
From the above discussion we infer that
\begin{equation}\label{fineq}
F_m>m.
\end{equation}
Note that 
\begin{equation}\label{f0}
F_0=n\pi
\end{equation}
is a transcendental number, since $\pi$ is transcendental and $n$ is rational. 
More generally, we have

\begin{prop}\label{prop:transcendental}
For every $m \in \mathbb{N}_0$ the number $F_m$ is transcendental. 
\end{prop}

\begin{proof}
In view of (\ref{f0}) we can assume without loss of generality
that $m \geq 1$. We argue by contradiction and assume that $F_m$ is algebraic. 
From (\ref{edef}) we obtain
\begin{equation}\label{feq}
\frac{m}{F_m}=\cos\bigg(\sqrt{\frac{F_m^2}{m^2}-1}-\frac{n\pi}{m}\bigg).
\end{equation}
Since $F_m>m$ by (\ref{fineq}), the assumption that $F_m$ is algebraic implies that
$i\sqrt{\tfrac{F_m^2}{m^2}-1}$ is an algebraic number different from zero. 
From Theorem~\ref{thm:Lindemann-Weierstrass} we conclude that $e^{i\sqrt{\frac{F_m^2}{m^2}-1}}$
is transcendental. Since $n$ is rational and $m$ is an integer, this
implies that 
$$\cos\bigg(\sqrt{\frac{F_m^2}{m^2}-1}-\frac{n\pi}{m}\bigg)
=\frac{1}{2}\Bigg(\frac{e^{i\sqrt{\tfrac{F_m^2}{m^2}-1}}}{
e^{i\frac{n\pi}{m}}}+\frac{e^{i\frac{n\pi}{m}}}{e^{i\sqrt{\tfrac{F_m^2}{m^2}-1}}}\Bigg)$$
is transcendental as well. But then by~\eqref{feq} the number $F_m$ is
transcendental, contradicting our assumption.
\end{proof}

If we assume Schanuel's conjecture we can say much more about these numbers. 

\begin{prop}\label{prop:schanuel}
Under the assumption that Schanuel's conjecture is true it follows that the  set
$\{F_m\mid m \in \mathbb{N}_0\}$ is algebraically independent.
\end{prop}

\begin{proof}
Under the assumption that Schanuel's conjecture holds true we will
show by induction that for every $N \in \mathbb{N}_0$ the set 
$$\mathfrak{F}_N:=\{F_m \mid 0 \leq m \leq N\}$$
is algebraically independent. For $N=0$ the set $\mathfrak{F}_0$ consists of
the single element $F_0=n\pi$ which is transcendental, and therefore
$\mathfrak{F}_0$ is algebraically independent. It remains to carry out
the induction step. For this purpose, we assume that
$\mathfrak{F}_{N-1}$ is algebraically independent 
and we want to conclude that $\mathfrak{F}_N$ is algebraically
independent. The strategy to prove this is the following. We abbreviate
$$
  G_m := \sqrt{\frac{F_m^2}{m^2}-1}
$$
and consider the set 
$$
  \mathfrak{G}_N := \{G_m\ \Big|\ 0 \leq m \leq N\}.
$$
We will first use the induction hypothesis to show that the set
$\mathfrak{G}_N$ is linearly independent over $\mathbb{Q}$, and then
conclude the induction step from Schanuel's conjecture in view of the
defining equation (\ref{edef}).  

Let us first show that $\mathfrak{G}_N$ is linearly independent over $\mathbb{Q}$.
We argue by contradiction and assume that this is not the case. By
induction hypothesis, $\mathfrak{F}_{N-1}$ is algebraically
independent. Since $G_m$ is an algebraic function of $F_m$ and vice
versa, it follows from Lemma~\ref{lem:trdeg} that the set
$\mathfrak{G}_{N-1}$ is algebraically independent as well. 
In particular, $\mathfrak{G}_{N-1}$ is linearly independent over
$\mathbb{Q}$. Hence, if $\mathfrak{G}_N$ is not linearly independent
over $\mathbb{Q}$, then there exist rational numbers $r_0,\ldots
r_{N-1} \in \mathbb{Q}$ such that 
\begin{equation}\label{eq:lindep}
  G_N=\sum_{m=0}^{N-1} r_mG_m.
\end{equation}
Since $\frac{m}{F_m} \in (0,1)$, 
we infer from (\ref{feq}) that
$$
  G_m-\frac{n\pi}{m} \in \Big(0,\frac{\pi}{2}\Big)
$$
and therefore
$$
  \sin\Big(G_m-\frac{n\pi}{m}\Big) \in (0,1).
$$
Hence we obtain from (\ref{feq}) the formula
\begin{equation}\label{eeq}
  e^{iG_m-\frac{in\pi}{m}} = \frac{m}{F_m}+i\sqrt{1-\frac{m^2}{F_m^2}}.
\end{equation}
From this we infer
\begin{eqnarray*}
  \frac{2}{\sqrt{\bigg(\sum_{m=0}^{N-1}r_m \sqrt{\frac{F_m^2}{m^2}-1}\bigg)^2
    +1}}
  &=& \frac{2N}{F_N}\\
  &=& 2\cos\Big(G_N-\frac{n\pi}{N}\Big)\\
  &=& \frac{e^{\frac{in\pi}{N}}}{e^{i\sum_{m=0}^{N-1}r_mG_m}}+\frac{e^{i\sum_{m=0}^{N-1}r_mG_m}}
{e^{\frac{in\pi}{N}}}\\
&=&\frac{e^{in\pi\big(\frac{1}{N}-\sum_{m=0}^{N-1}\frac{r_m}{m}\big)}}{
\prod_{m=0}^{N-1}\bigg(\frac{m}{F_m}+i\sqrt{1-\frac{m^2}{F_m^2}}
\bigg)^{r_m}}\\
& &+\frac{\prod_{m=0}^{N-1}\bigg(\frac{m}{F_m}+i\sqrt{1-\frac{m^2}{F_m^2}}
\bigg)^{r_m}}{e^{in\pi\big(\frac{1}{N}-\sum_{m=0}^{N-1}\frac{r_m}{m}\big)}}
\end{eqnarray*}
Here the first equality follows from~\eqref{eq:lindep} and the
definition of $G_m$,
the second one from~\eqref{feq} for $m=N$,
the third one from Euler's formular and~\eqref{eq:lindep}, 
and the fourth one from~\eqref{eeq}. 
This shows that the set $\mathfrak{F}_{N-1}$ is algebraically
dependent, in contradiction to the induction hypothesis. Hence the set
$\mathfrak{G}_N$ is linearly independent over $\mathbb{Q}$.

It follows that the set
$$
  i\mathfrak{G}_N = \{iG_m\mid 0 \leq m \leq N\}
$$
is linearly independent over $\mathbb{Q}$ as well.
By~\eqref{eeq}, for each $0\leq m\leq N$ the number $e^{iG_m}$ is
algebraic over $\Q(\mathfrak{F}_N)$ and thus over $\Q(i\mathfrak{G}_N)$.
Therefore, from Lemma~\ref{lem:trdeg} we conclude that
$$
  \mathrm{trdeg}\,\mathbb{Q}\big(iG_0,\ldots iG_N,e^{iG_0},\ldots
  e^{iG_N}\big) = \mathrm{trdeg}\,\mathbb{Q}\big(iG_0,\ldots,iG_N\big).
$$
Since we assume that Schanuel's conjecture holds true, this implies that
$$
  \mathrm{trdeg}\,\mathbb{Q}\big(iG_0,\ldots,iG_N\big) \geq N+1.
$$
Hence the set $i\mathfrak{G}_N$, and therefore also the set
$\mathfrak{F}_N$, is algebraically independent. This proves the proposition. 
\end{proof}

Propositions~\ref{prop:transcendental} and~\ref{prop:schanuel}
together give Theorem~\ref{thm:schanuel} from the Introduction.

\begin{rem}
It is interesting to observe that the action
values~\eqref{eq:action-billiard} of the disk billiard are all
algebraic multiples of $R\sqrt{E}$, and the determine via
Proposition~\ref{prop:convex-toric} the EBK spectral values that are
algebraically independent (assuming Schanuel's conjecture). This is
possible because the formulas in Proposition~\ref{prop:convex-toric}
involve a limiting process (sup or inf). 
\end{rem}

\begin{rem}
The quantum mechanical spectrum of the disk billiard is given by
zeroes of higher order Bessel functions~\cite{Weibert}. In view of
Proposition~\ref{prop:schanuel}, one may wonder whether these numbers
are also algebraically independent.   
\end{rem}

\section{Maslov shifts}\label{sec:Maslov}

In this section we discuss the effect of Maslov shifts on the various
spectra. 
For this, we replace the quantization condition~\eqref{eq:quant} on a
torus $L$ by
\begin{equation*}
  \int_{\gamma_j}\lambda\in 2\pi\hbar(\Z+\mu_j)\quad \text{for }j=1,\dots,n.
\end{equation*}
Here $\gamma_1,\dots,\gamma_n$ is a basis of $H_1(L)$ and we are
fixing a vector (representing $1/4$ times the Maslov shifts)
$$
  \mu=(\mu_1,\dots,\mu_n)\in\R^n.
$$
The shifted EBK spectrum is defined as before, using this quantization
condition. For a toric Hamiltonian as in~\S\ref{ss:toric} the
quantization condition becomes $p_j=\hbar(m_j+\mu_j)$ and the EBK
spectrum in~\eqref{eq:toric-EBK} becomes 
\begin{equation*}
  \sigma_\EBK^\mu = \{f(\hbar(m+\mu))\mid m\in\N_0^n\}.
\end{equation*}
The shifted EKB spectrum of the uncoupled harmonic oscillators in
Example~\ref{ex:harm} becomes
$$
  \sigma_\EBK^\mu = \Bigl\{\sum_{j=1}^n\hbar\om_j (m_j+\mu_j) \mid m\in\N_0^n\Bigr\}.
$$
which for $\mu_j=1/2$ agrees with the quantum mechanical spectrum. 

We shift the marked action spectrum in~\S\ref{ss:toric} to
$$
  \A^\mu = \{(k,\la p+\mu,k\ra) \mid k\in\Z^n,\,p\in N,\,n(p)\sim k\}
$$
and use this to define
$$
  M_\mu := \overline{\Bigl\{\frac{k}{a}\;\Bigl|\;(k,a)\in\A^\mu,\,a\neq 0\Bigr\}}.
$$
as in Proposition\ref{prop:action-EBK}. Since
\begin{align*}
  \Bigl\{\frac{k}{a}\;\Bigl|\;(k,a)\in\A^\mu\Bigr\}
  &= \Bigl\{\frac{k}{\la p+\mu,k\ra}\;\Bigl|\;k\in\Z^n,\,p\in N,\,n(p)\sim k\Bigr\}\cr
  &= \Bigl\{\frac{k}{\la p,k\ra}\;\Bigl|\;k\in\Z^n,\,p\in N+\mu,\,n(p)\sim k\Bigr\},
\end{align*}
the formula for recovering $N$ from the marked action spectrum in
Proposition\ref{prop:action-EBK} becomes
$$
  N+\mu = \overline{L(M_\mu')}.
$$   
Corollary~\ref{cor:action-EBK} remains true with the shifted EBK spectrum
$$
  \sigma_\EBK^\mu = \Bigl\{E\in\R\;\Bigl|\; \text{there exists }m\in \N_0^n\text{ such
    that }\frac{\hbar(m+\mu)}{E}\in N\Bigr\}.
$$
Formula~\eqref{eq:Em} in Proposition~\ref{prop:convex-toric} becomes
\begin{equation*}
  E_m^\mu = \sup_{(k,a)\in\A}\frac{\hbar\la m+\mu,k\ra}{a}
  = \sup_{k\in\N_0^n}\frac{\hbar\la m+\mu,k\ra}{a(k)},\qquad m\in\N_0^n.
\end{equation*}
For the disk billiard in~\S\ref{sec:billiard} we choose the shifts $0$
and $3/4$ for the loops $\gamma$ and $\delta$, respectively, so the
EBK quantization conditions become 
$$
  \int_{\gamma}\lambda=2\pi\hbar m,\quad\text{and}\quad
  \int_{\delta}\lambda=2\pi\hbar \Bigl(n+\frac34\Bigr),\qquad m,n\in\N_0.
$$
Then equation~\eqref{eq:billiard-EBK} gets replaced by
\begin{equation*}
  \sqrt{F^2-m^2} - m\arccos(m/F) = \pi(n+3/4),
\end{equation*}
which for $F=kR$ agrees with formula (4.19) in~\cite{Weibert}. This
formula is reproduced by the approach in~\S\ref{sec:disk-toric}
with the shifts $\mu_1=\mu_2=3/4$, i.e.~replacing $m_j$ by $m_j+3/4$
for $j=1,2$.


 


\end{document}